\documentclass[12pt]{article}
\usepackage{amssymb,amsthm,amsmath,amsfonts,latexsym,tikz,hyperref,shuffle,ytableau}
\usepackage[hmargin=1in,vmargin=1in]{geometry}
\usepackage{enumitem}

\newcommand{\ben}{\begin{enumerate}}
\newcommand{\een}{\end{enumerate}}
\newcommand{\ble}{\begin{lem}}
\newcommand{\ele}{\end{lem}}
\newcommand{\bth}{\begin{thm}}
\renewcommand{\eth}{\end{thm}}
\newcommand{\bpr}{\begin{prop}}
\newcommand{\epr}{\end{prop}}
\newcommand{\bco}{\begin{cor}}
\newcommand{\eco}{\end{cor}}
\newcommand{\bcon}{\begin{conj}}
\newcommand{\econ}{\end{conj}}
\newcommand{\bde}{\begin{defn}}
\newcommand{\ede}{\end{defn}}
\newcommand{\bex}{\begin{exa}}
\newcommand{\eex}{\end{exa}}
\newcommand{\barr}{\begin{array}}
\newcommand{\earr}{\end{array}}
\newcommand{\btab}{\begin{tabular}}
\newcommand{\etab}{\end{tabular}}
\newcommand{\beq}{\begin{equation}}
\newcommand{\eeq}{\end{equation}}
\newcommand{\bea}{\begin{eqnarray*}}
\newcommand{\eea}{\end{eqnarray*}}
\newcommand{\bal}{\begin{align*}}
\newcommand{\bce}{\begin{center}}
\newcommand{\ece}{\end{center}}
\newcommand{\bpi}{\begin{picture}}
\newcommand{\epi}{\end{picture}}
\newcommand{\bpp}{\begin{picture}}
\newcommand{\epp}{\end{picture}}
\newcommand{\bfi}{\begin{figure} \begin{center}}
\newcommand{\efi}{\end{center} \end{figure}}
\newcommand{\bprf}{\begin{proof}}
\newcommand{\eprf}{\end{proof}\medskip}

\newcommand{\bsl}{\begin{slide}{}}
\newcommand{\esl}{\end{slide}}
\newcommand{\bfr}{\begin{frame}}
\newcommand{\efr}{\end{frame}}

\newcommand{\comp}{\models}

\newcommand{\hqed}{\hfill \qed}

\newcommand{\eqed}[1]{$\textcolor{white}{\qed}\hfill{\dil#1}\hfill\qed$}
\newcommand{\eqqed}[1]{$\rule{1ex}{0ex}\hfill{\dil#1}\hfill\qed$}

\newcommand{\ol}{\overline}

\newcommand{\hso}[1]{\hspace{-1pt}}
\newcommand{\vs}[1]{\vspace{#1}}
\newcommand{\qmq}[1]{\quad\mbox{#1}\quad}

\newcommand{\emp}{\emptyset}

\newcommand{\sbs}{\subset}
\newcommand{\sbe}{\subseteq}

\newcommand{\spe}{\supseteq}

\newcommand{\ptn}{\vdash}

\newcommand{\jn}{\vee}

\def\<{\langle}
\def\>{\rangle}

\newcommand{\ree}[1]{(\ref{#1})}

\newcommand{\ra}{\rightarrow}

\newcommand{\al}{\alpha}
\newcommand{\be}{\beta}
\newcommand{\ga}{\gamma}
\newcommand{\de}{\delta}

\newcommand{\io}{\iota}
\newcommand{\ka}{\kappa}
\newcommand{\la}{\lambda}

\newcommand{\si}{\sigma}

\newcommand{\ze}{\zeta}

\newcommand{\bx}{{\bf x}}

\newcommand{\bbR}{{\mathbb R}}

\newcommand{\fS}{{\mathfrak S}}

\DeclareMathOperator{\Av}{Av}

\DeclareMathOperator{\Des}{Des}

\DeclareMathOperator{\sh}{sh}

\DeclareMathOperator{\SYT}{SYT}

\newcommand{\dil}{\displaystyle}

\newtheorem{thm}{Theorem}[section]
\newtheorem{prop}[thm]{Proposition}
\newtheorem{cor}[thm]{Corollary}
\newtheorem{lem}[thm]{Lemma}
\newtheorem{conj}[thm]{Conjecture}
\newtheorem{exa}[thm]{Example}

\newcommand{\lab}{\ol{\la}}

\newcommand{\fSb}{\ol{\fS}}

\newcommand{\fR}{{\mathfrak R}}

\DeclareMathOperator{\std}{std}

\DeclareMathOperator{\lis}{lis}
\DeclareMathOperator{\lds}{lds}

\DeclareMathOperator{\iDes}{iDes}
\DeclareMathOperator{\RS}{RS}

\usepackage{accents}
\newlength{\dhatheight}

\newcommand{\overdoublerightarrow}[1]{\stackrel{\twoheadrightarrow}{#1}}    
\newcommand{\overdoubleleftarrow}[1]{\stackrel{\twoheadleftarrow}{#1}}

\usepackage{mathtools}

\begin{document}
\pagestyle{plain}

\title{Revisiting pattern avoidance and quasisymmetric functions
}
\author{Jonathan Bloom\\[-5pt]
\small Department of Mathematics, Lafayette College,\\[-5pt]
\small Easton, PA 18042, USA, {\tt bloomjs@lafayette.edu}
and\\
Bruce E. Sagan\\[-5pt]
\small Department of Mathematics, Michigan State University,\\[-5pt]
\small East Lansing, MI 48824, USA, {\tt sagan@math.msu.edu}
}

\date{\today\\[10pt]
	\begin{flushleft}
	                                       \\[5pt]
	\end{flushleft}}

\maketitle

\begin{abstract}
Let $\fS_n$ be the $n$th symmetric group.  Given a set of permutations $\Pi$ we denote by $\fS_n(\Pi)$ the set of permutations in $\fS_n$ which avoid $\Pi$ in the sense of pattern avoidance.  Consider the generating function $Q_n(\Pi)=\sum_\si F_{\Des\si}$ where the sum is over all $\si\in\fS_n(\Pi)$ and $F_{\Des\si}$ is the fundamental quasisymmetric function corresponding to the descent set of 
$\si$.  Hamaker, Pawlowski, and Sagan introduced $Q_n(\Pi)$ and studied its properties, in particular, finding criteria for when this quasisymmetric function is symmetric or even Schur nonnegative for all $n\ge0$.  The purpose of this paper is to continue their investigation answering some of their questions, proving one of their conjectures, as well as considering other natural questions about $Q_n(\Pi)$.  In particular we look at $\Pi$ of small cardinality, superstandard hooks, partial shuffles, Knuth classes, and a stability property.
\end{abstract}

\section{Introduction}

Let $\fS_n$ denote the symmetric group of all permutations $\pi=\pi_1\pi_2\dots\pi_n$ of the set $[n]:=\{1,2,\dots, n\}$.  We sometimes insert commas between the elements of $\pi$ or enclose the permutation in parentheses to improve readability.  We also use the notation 
$[m,n]=\{m,m+1,\dots,n\}$.   Given any sequence of distinct real numbers $\si$ its {\em standardization}, $\std\si$, is the permutation obtained by replacing its smallest element by $1$, its next smallest by $2$, and so forth.  We say that $\si\in\fS_n$ {\em contains
$\pi\in\fS_k$ as a pattern} if there is some subsequence $\si'$ of $\si$ with $\std\si'=\pi$.  If no such subsequence then $\si$ {\em avoids} $\pi$.  For a set of permutations $\Pi$ we let
$$
\fS_n(\Pi)= \{\si\in\fS_n \mid \text{$\si$ avoids every $\pi\in\Pi$}\}
$$
and
$$
\fSb_n(\Pi)=\fS_n-\fS_n(\Pi)= \{\si\in\fS_n \mid \text{$\si$ contains some $\pi\in\Pi$}\}.
$$
We omit the set braces in $\Pi$ if it contains only one permutation.  For example, $\si=25143$ contains $\pi=132$ because 
$\std 254=132$ but $\si\in\fS_n(123)$ since $\si$ contains no increasing subsequence with three elements.
More information about pattern avoidance can be found in the book of B\'ona~\cite{bon:cp}.

Let $\bx=\{x_1,x_2,\dots\}$ be a countably infinite set of variables.  An element of the formal power series ring $\bbR[[\bx]]$ is a {\em symmetric function} if it is of bounded degree and invariant under permutations of the variables.  Bases for the vector space of symmetric functions homogeneous of degree $n$ are indexed by integer partitions $\la=(\la_1,\la_2,\dots,\la_l)$ of $n$ which we denote 
$\la\ptn n$.  We use Greek letters near the middle of the alphabet to denote partitions and also use the multiplicity notation $i^{k_i}$ if a part $i$ of $\la$ is repeated $k_i$ times.   In particular, we will be interested in the basis $m_\la$ of monomial symmetric functions which is obtained by symmetrizing the monomial $x_1^{\la_1} x_2^{\la_2} \dots x_l^{\la_l}$, as well as the Schur functions $s_\la$ about which we say more below.  As an example
$$
m_{(2,1)}=x_1^2 x_2 + x_2^2 x_1 + x_1^2 x_3 + x_3^2 x_1  + x_2^2 x_3 + x_3^2 x_2 + \cdots .
$$
For information about symmetric functions as well as related material concerning Young tableaux and the Robinson-Schensted correspondence  (which we use throughout) the reader can consult the texts of Sagan~\cite{sag:sym} or Stanley~\cite{sta:ec2}.

An element of $\bbR[[\bx]]$  is {\em quasisymmetric} if it is invariant under bijections between subsets of the variables which preserve the order of the subscripts.  
The {\em quasisymmetric functions} are those power series which are quasisymmetric and of bounded degree.  The were first explicitly introduced by Gessel~\cite{ges:mpi} and have since found many applications, see~\cite{sta:ec2}.  Bases for the vector space of quasisymmetric functions of degree $n$ are indexed by composition (ordered partitions) $\al=(\al_1,\al_2,\dots,\al_l)$ of $n$ and we use the notation $\al\comp n$ as well as multiplicity notation.  To distinguish them from partitions, we use letters from the beginning of the Greek alphabet for compositions.  The monomial quasisymmetric $M_\al$ function is gotten by quasisymmetrizing the monomial
$x_1^{\al_1} x_2^{\al_2} \dots x_l^{\al_l}$, for example
$$
M_{(1,2)} = x_1 x_2^2 + x_1 x_3^2 + x_2 x_3^2 + \cdots.
$$
Note that 
\beq
\label{m-M}
m_\la=\sum_\al M_\al
\eeq
where the sum is over all $\al$ obtained by rearranging the parts of $\la$.

There is another important basis for the quasisymmetric functions which we will be using.  To define it, note that there is a bijection between compositions $\al\comp n$ and subsets $S\sbe [n-1]$ given by
\beq
\label{al-S}
(\al_1,\al_2,\dots,\al_l)\mapsto \{\al_1,\ \al_1+\al_2,\ \dots,\ \al_1+\al_2+\dots+\al_{\ell-1}\}.
\eeq
The {\em fundamental quasisymmetric function} associated with $S\sbe[n-1]$ is
$$
F_S = \sum x_{i_1} x_{i_2} \dots x_{i_n}
$$
where the sum is over indices satisfying $i_1\le i_2\le\dots\le i_n$ and $i_j<i_{j+1}$ if $j\in S$.  To illustrate, if $S=\{1\}\sbe[2]$ then
$$
F_S =  x_1 x_2^2 + x_1 x_3^2 + x_2 x_3^2 + \dots + x_1 x_2 x_3 + x_1 x_2 x_4 + x_1 x_3 x_4+ x_2 x_3 x_4 + \cdots.
$$
We also denote $F_S$ by $F_\al$ if $S$ corresponds to $\al$  under the bijection above.  The expansion of a fundamental quasisymmetric function in terms of monomials is
\beq
\label{F-M}
F_\al=\sum_{\be\le\al} M_\be
\eeq
where $\be\le\al$ means that $\be$ is a refinement of $\al$.  In the example above we see that $F_{(1,2)}=M_{(1,2)}+M_{(1^3)}$.

We study certain quasisymmetric functions related to pattern avoidance which were introduced by Hamaker, Pawlowski, and Sagan~\cite{hps:paq}.  
Related work has been done by Adin and Roichman~\cite{ar:mcd} and by Elizalde and Roichman~\cite{er:ap,er:sap}.
A permutation $\si\in\fS_n$ has {\em descent set}
$$
\Des\si = \{i \mid \si_i>\si_{i+1}\}\sbe [n-1].
$$
Given a set of permutations $\Pi$, define
$$
Q_n(\Pi)=\sum_{\si\in\fS_n(\Pi)} F_{\Des\si}.
$$
In~\cite{hps:paq} they found many interesting $\Pi$ such that  for all $n$ the function $Q_n(\Pi)$ is symmetric.  In that case, they were also often able to show that $Q_n(\pi)$ is {\em Schur nonnegative} in that the coefficients of its expansion in the Schur basis are nonnegative.  Our main motivation for the present work is to answer some of the questions asked by Hamaker-Pawlowski-Sagan and to prove one of their conjectures.   

Our work will be simplified by using certain symmetries of permutations.  A permutation $\pi=\pi_1\pi_2\dots\pi_k$ has {\em reversal} 
$\pi^r=\pi_k\dots\pi_2\pi_1$ and {\em  complement} $\pi^c=k+1-\pi_1,k+1-\pi_2,\dots,k+1-\pi_k$.  
We apply these operations to sets of partitions by applying them to each individual element of the set.
Also, given a partition
 $\la=(\la_1,\dots,\la_l)$ we use the notation $\la^t= (\la_1^t,\dots,\la_m^t)$ for its {\em transpose} given be reflecting the Young diagram for $\la$ across the diagonal.  We write our Young diagrams in English notation with the first row on the top.  We also give coordinates to elements of a Young diagram as in a matrix.
\bpr[\cite{hps:paq}]
\label{rc}
If $Q_n(\Pi)$ is symmetric then so are $Q_n(\Pi^r)$ and $Q_n(\Pi^c)$.  In particular, if $Q_n(\Pi)=\sum_\la c_\la s_\la$ for certain coefficients $c_\la$ then

\vs{10pt}

\eqed{
Q_n(\Pi^r)=Q_n(\Pi^c)=\sum_\la c_\la s_{\la^t}.
}
\epr

We make heavy use of  the following result of Gessel~\cite{ges:mpi}.  Suppose that $P$ is a {\em standard Young tableau (SYT) of shape $\la\ptn n$}, that is, a filling of the Young diagram of $\la$ with the numbers in $[n]$ so that rows and columns increase.  
We indicate that $P$ has shape $\la$ by writing $\sh P=\la$.
Let 
$$
\SYT(\la)=\{P \mid \sh P =\la\}
$$
and $f^\la=\#\SYT(\la)$ where the hash sign denotes cardinality.  The {\em descent set} of $P$ is
$$
\Des P = \{ i \mid \text{$i+1$ appears in a lower row than $i$ in $P$}\}.
$$
This notion permits one to expand the Schur functions in terms of the fundamental quasisymmetric functions.
\bth[\cite{ges:mpi}]
\label{ges}
For any $\la$ we have

\vs{10pt}

\eqqed{
s_\la = \sum_{Q\in\SYT(\la)} F_{\Des Q}.
}
\eth

Certain properties of the Robinson-Schensted correspondence will be crucial.  We only review the ones we need here and the interested reader can find more detail in~\cite{sag:sym,sta:ec2}.  The Robinson-Schensted map is a bijection
$$
\RS:\fS_n\ra\bigcup_{\la\ptn n} \SYT(\la)\times\SYT(\la).
$$
If $\RS(\pi)=(P,Q)$ then we write $P=P(\pi)$ and $Q=Q(\pi)$ and call $P$ and $Q$ the {\em $P$-tableau} and {\em $Q$-tableau} of $\pi$, respectively.
We need the following properties of the map $\RS$.
\bth
\label{RS}
Suppose $\RS(\pi)=(P,Q)$.
\ben
\item[(a)] $\Des\pi=\Des Q$.
\item[(b)] If $\sh P=\la$ then $\la_1$ is the length of a longest increasing subsequence of $\pi$.
\item[(c)] $P(\pi^r) = (P(\pi))^t$.
\item[(d)] $\RS(\pi^{-1}) = (Q,P)$. \hqed
\een
\eth

Call two permutations $\pi,\si$ {\em Knuth equivalent} if $P(\pi)=P(\si)$.  Given an SYT denote by $K(P)$ the Knuth equivalence class of all permutations with $P(\pi)=P$.  Similarly, if $\la$ is a partition we let 
$$
K(\la)=\{\pi \mid \sh P(\pi) =\la\}.
$$
The following result follows easily from Theorem~\ref{ges} and Theorem~\ref{RS} (a).
\bco[\cite{hps:paq}]
\label{knuth}
Suppose $P\in\SYT(\la)$.
\ben
\item[(a)] $\sum_{\pi\in K(P)} F_{\Des\pi} = s_\la$.
\item[(b)] $\sum_{\pi\in K(\la)} F_{\Des\pi}  = f^\la s_\la$. \hqed
\een
\eco

The rest of this paper is structured as follows.  In the next section we determine which $\Pi$ of cardinality $\#\Pi\le 2$ have $Q_n(\Pi)$ symmetric for all $n$.  
The following section answers a question of Hamaker-Pawlowski-Sagan concerning the coefficients in the Schur expansion of $Q_n(K(P))$ where $K(P)$ is the Knuth class of a superstandard hook tableau $P$.
Section~\ref{ps} is devoted to proving a conjecture  in~\cite{hps:paq} about $Q_n(\Pi)$ where $\Pi$ is a certain variant of a shuffle set called a partial shuffle.  In Section~\ref{pkc} we study $\Pi$ such that $\fS_n(\Pi)$ is a union of Knuth classes for all $n$ (which implies that $Q_n(\Pi)$ is Schur nonnegative).  In so doing, we provide a simpler proof of a theorem in~\cite{hps:paq} when $\Pi=K(P)$ for a single SYT $P$ and also answer a question asked by the authors about the case when $\Pi$ is a union of two Knuth classes.  We end with a section about stability results.

\section{Pattern sets of small size}
\label{pss}

In this section we answer the question: for which $\Pi\sbe\fS_k$ with $\#\Pi=1$  or $2$ is $Q_n(\Pi)$ symmetric for all $n$?  It turns out that this occurs exactly when $\Pi\sbe \{\io_k, \delta_k\}$ where $\io_k$ and $\de_k$ are the increasing and decreasing permutations in $\fS_k$, respectively.  We need the following result which  follows easily from Theorem~\ref{RS} (b) and (c) and Corollary~\ref{knuth} (b).
\ble[\cite{hps:paq}]
\label{iode}
For $n\ge0$ we have
\begin{align*}
Q_n(\emp)&=\sum_{\la} f^\la s_\la,\\
Q_n(\io_k)&=\sum_{\la_1<k} f^\la s_\la,\\
Q_n(\de_k)&=\sum_{\la_1^t<k} f^\la s_\la,
\end{align*}
where all three sums are over $\la\ptn n$ together with any additional restriction noted in the summation.\hqed
\ele

We now turn to the case $\#\Pi=1$. 

\bth
\label{Pi=1}
Suppose $\#\Pi=1$.  Then $Q_n(\Pi)$ is symmetric if and only if $\Pi=\{\io_k\}$ or $\Pi=\{\de_k\}$ for some $k$.
\eth
\bprf
The reverse implication follows from Proposition~\ref{iode}.  For the other direction suppose, towards a contradiction, that $Q_n(\pi)$ is symmetric for all $n$ but $\pi\neq\io_k, \de_k$.  Since $\fS_k(\pi) = \fS_k-\{\pi\}$ we have
$$
Q_k(\pi) = Q_k(\emp)-F_{\Des\pi}.
$$
From our assumption and the previous lemma we have $Q_k(\pi)$ and $Q_k(\emp)$ are symmetric, so the same must be true of $F_{\Des\pi}$ where $\Des\pi\sbe[k-1]$.  But $\pi\neq\io_k, \de_k$ so we must have $\Des\pi\neq\emp,[k-1]$.  It follows that there are $a,b\in [k-1]$ such that $a\in\Des\pi$ and $b\not\in\Des\pi$.  Now translate this statement  from subsets of $[k-1]$ to compositions of $k$ using~\ree{al-S}: if $\Des\pi$ corresponds to a composition $\al$ then there are two compositions $\be,\ga$ which are both rearrangements of the partition $\la=(2,1^{k-2})$ such that $\be$ refines $\al$ but $\ga$ does not.  Considering the expansion of 
$F_\al$ into monomial quasisymmetric functions in~\ree{F-M}, we see from~\ree{m-M} that we have some but not all the terms which would be needed to give the monomial symmetric function $m_\la$.  This contradiction completes the proof.
\eprf

We now consider the case $\#\Pi=2$.
\ble
\label{Pi=2spec}
If $\Pi=\{\io_k,\de_l\}$ for any $k,l\ge1$ then $Q_n(\Pi)$ is symmetric and Schur nonnegative for all $n$.
\ele
\bprf
We have
$$
\fSb_n=\fSb_n(\io_k)\cup\fSb_n(\de_l)-\{\si\in\fS_n \mid \text{$\si$ contains both $\io_k$ and $\de_l$}\}.
$$
We know from Lemma~\ref{iode} that the generating function for the first two sets on the right-hand side are symmetric and Schur nonnegative.  So it suffices to prove that the same is true of the third.  But by Theorem~\ref{RS} (b) and (c)
$$
\{\si\in\fS_n \mid \text{$\si$ contains both $\io_k$ and $\de_l$}\}=\biguplus_\la K(\la)
$$
where the disjoint union is over all  $\la\ptn n$ satisfying $\la_1\ge k$ and $\la_1^t\ge l$.  Applying Corollary~\ref{knuth} (b) completes the proof.
\eprf

For the next result we make use of the lattice $C_k$ of compositions of $k$ ordered by refinement.  So $\al\le\be$ and other notations refer to this partial order.  We also sometimes use the notation $M(\al)$ for $M_\al$ for readability.
\bth
Suppose $k\ge4$ and $\Pi\sbe\fS_k$ with $\#\Pi=2$.  Then $Q_n(\Pi)$ is symmetric for all $n$ if and only if $\Pi=\{\io_k,\de_k\}$.
\eth
\bprf
The backwards direction follows from the previous lemma.  For the forward direction, it is easy to check by computer that this is true for $k=4$, so we assume that $k\ge5$.   There are now two cases depending on 
$|\Pi\cap \{\io_k,\de_k\}|$.   

First consider $\Pi=\{\pi,\de_k\}$ where $\pi\in\fS_k-\{\io_k\}$.  The other possibility when $|\Pi\cap \{\io_k,\de_k\}|=1$ is handled similarly.  If $Q_k(\Pi)$ is symmetric then so is $F_{\Des\pi}+F_{[k-1]}= F_{\Des\pi}+s_{1^k}$.  It follows that $F_{\Des\pi}$ is symmetric.  But then $\pi\in\{\io_k,\de_k\}$ by Theorem~\ref{Pi=1}, which contradicts our choice of $\Pi$.

Now assume that $\Pi=\{\pi,\si\}$ with $\Pi\cap \{\io_k,\de_k\}=\emp$.  As in the previous paragraph, the fact that $Q_k(\Pi)$ is symmetric implies that so if $F_\al+F_\be$ where $\al=\Des\pi$ and $\be=\Des\si$.  Label the atoms of $C_k$ as 
$\al_i=1^{i-1} 2 1^{k-i-1}$ for $1\le i< k$.  Expanding the fundamental quasisymmetric function sum  in terms of monomial quasisymmetrics we see that $\sum M(\al_i) + \sum M(\al_j)$ is symmetric where the two sums are over the sets define by
$$
A=\{\al_i \mid \al_i\le\al\}  \qmq{and} B=\{\al_j \mid \al_j\le\be\}.
$$
Symmetry and the fact that $\pi,\si\neq\io_k,\de_k$ imply that $A$ and $B$ are disjoint, nonempty, and $A\uplus B=\{\al_1,\dots,\al_{k-1}\}$.

Without loss of generality we can assume that $\al_1\in A$.  Since $B\neq\emp$, there is an  index $i$ such that $\al_i\in A$ and 
$\al_{i+1}\in B$.  Let $i$ be the minimum such index.
If $i\ge2$ then $\al_1,\al_2\in A$ and so $M(\al_1\jn\al_2) = M(3 1^{k-3})$ is in the expansion of $F_\al$.  But $\al_i\in A$ and 
$\al_{i+1}\in B$ which implies that $M(\al_i\jn\al_{i+1}) = M(1^{i-1} 3 1^{k-i-2})$ is not in the expansion of $F_\al+F_\be$.  It follows that this expansion is not symmetric which is a contradiction.  So it must be that $i=1$. 

We have shown that $\al_1\in A$ implies $\al_2\in B$.  Similarly, $\al_2\in B$ implies $\al_3\in A$, and so forth.  It follows that
$$
A=\{\al_1,\al_3,\dots\} \qmq{and} B=\{\al_2,\al_4,\dots\}.
$$
Since $k\ge5$ we have $\#A\ge2$ and $\#B\ge2$.  Thus  $M(\al_1\jn\al_3)=M(2^21^{k-4})$ is in the expansion of $F_\al$.
But $M(\al_1\jn\al_4)=M(2121^{k-5})$ is not in the expansion of $F_\al+F_\be$.  This final contradiction shows that there is no $\Pi$ of this form with $Q_n(\Pi)$ being symmetric for all $n$.
\eprf

We note that this result is not true when $k=3$.  For example, $\Pi=\{213,231\}$ has $Q_n(\Pi)$ symmetric and Schur nonnegative for all $n$.  We also conjecture, based on computer experiments, that things change when $\#\Pi\ge3$.
\bcon
Given $p\ge3$ there is a $K$ which is a function of $p$ such that if $\#\Pi=p$ and $\Pi\sbe\fS_k$ for $k\ge K$ then $Q_n(\Pi)$ can not be symmetric for all $n$.
\econ

\section{Superstandard hooks}

An SYT $P$ of shape $\la$ is called {\em row superstandard} if it is obtained by filling the first row with the integers in $[1,\la_1]$, then the second row with the entries $[\la_1+1,\la_1+\la_2]$, and so on.  {\em Column superstandard} is defined analogously using the columns.  {\em Superstandard} means either row or column superstandard.  A {\em superstandard hook} is a superstandard tableau of hook shape. For us, these tableaux are interesting because of the following result.

\bth[\cite{hps:paq}]
\label{ss}
For any SYT $P$ we have $\fS_n(K(P))$ is a union of Knuth classes for all $n$ if and only if $P$ is a superstandard hook.
\hqed
\eth

We now answer Questions~5.13 from~\cite{hps:paq}.   Specifically, we know from the previous result and Corollary~\ref{knuth} (a) that for a superstandard hook $P$ the generating function $Q_n(K(P))$ is Schur nonnegative.  So what do the coefficients in its Schur expansion count?  By Proposition~\ref{rc} it suffices to consider the row superstandard case.
Given positive integers $r,s$ and some SYT $P$ then an {\em $(r,s)$-ascending sequence} in $P$ is a sequence of its elements
\beq
\label{rsa}
p_{1,r}=p_{i_1,j_1}<p_{i_2,j_2}<\dots <p_{i_s,j_s}
\eeq
such that $1=i_1<i_2<\dots<i_s$.  We let
$$
\Av_n(r,s)=\{ P\in\SYT(n) \mid \text{$P$ does not contain an $(r,s)$-ascending sequence}\}.
$$
\bth\label{thm:hook expansion}
Let $R$ be the row superstandard tableau of shape $(r,1^{s-1})$.  Then
$$
Q_n(K(R))=\sum_\la c_\la s_\la
$$
where 
$$
c_\la=\text{the number of $P\in\Av_n(r,s)$ of shape $\la$}.
$$
\eth
\bprf
By Corollary~\ref{knuth} (a) and Theorem~\ref{ss}, it suffices to show that $P\in\SYT(n)$ contains an $(r,s)$-ascending sequence if and only if some permutation with insertion tableau $P$ contains an element of $K(R)$ as a pattern.  Note that
$$
K(R) = [(1,2,\dots,r-1)\shuffle(r+s-1,r+s-2,\dots,r+1)] \cdot r
$$
where the multiplication sign denotes concatenation.

First consider the forward direction and suppose $P$ contains an $(r,s)$-ascending sequence as in~\ree{rsa}.  Consider the row word $\rho$  of $P$.  Because of the restriction on the first coordinates of the subscripts we see that 
$p_{i_s,j_s}>\dots>p_{i_1,j_1}=p_{1,r}$
is a subsequence of $\rho$. Furthermore, the elements $p_{1,1}<p_{1,2}<\dots<p_{1,r-1}$ come in that order before $p_{1,r}$ in 
$\rho$.  The union of these two subsequences standardizes to an element of $K(R)$ which is what we wished to prove.

For the converse, consider the column word $\ka$ of $P$.  So $\ka=C_1 C_2\dots C_t$ where $C_j$ is the $j$th column of $P$ read in decreasing order.  Let $\pi$ be a copy of some element of $K(R)$ in $\ka$.  So $\pi$ contains a decreasing subsequence of length $s$, say $p_{i_1,j_1}>\dots>p_{i_s,j_s}$.  We claim that $i_k>i_{k+1}$ for all $k$.  For assume to the contrary that $i_k\le i_{k+1}$.  But also $j_k\le j_{k+1}$ by the column ordering in $P$. This  forces $p_{i_k,j_k}<p_{i_{k+1},j_{k+1}}$ which is a contradiction.
Furthermore, $p_{i_s,j_s}$  must be the end of an increasing subsequence of $\pi$ of length $r$.  Since $\ka$ lists columns in decreasing order, this forces $j_s\ge r$.  Since $p_{1,r}$ is the minimum element in the columns weakly right of column $r$ we have the sequence
$$
p_{1,r}<p_{i_{s-1},j_{s-1}}<\dots<p_{i_1,j_1}
$$
which is the desired $(r,s)$-ascending sequence in $P$.
\eprf

\section{Partial shuffles}
\label{ps}

The goal of this section is to prove Conjecture~4.2 in \cite{hps:paq}. We first establish some definitions.   For any $a\leq n $ define the corresponding \emph{partial shuffle} as 
$$(1,2,\ldots,\widehat{a}, \ldots, n) \cshuffle a = [(1,2,\ldots,\widehat{a}, \ldots, n) \shuffle a] -\iota_n.$$  
where $\shuffle$ denotes the standard shuffle and $\ \widehat{}\ $ denotes deletion of the indicated element.  We also define the set of \emph{fattened hooks} to be 
$$H_{n,k}= \{\lambda\vdash n \ |\  (2,k)\notin \lambda\}$$
so that $H_{n,2}$ is the set of ordinary hooks.
Set 
$$\Pi_{k,2} = (1,2,\ldots,k, k+1,\widehat{k+2})\cshuffle (k+2)$$
and 
$$\Pi_{k,1} = (1,2,\ldots, k,\widehat{k+1},k+2)\cshuffle (k+1).$$
We now prove the following theorem where the second equality was first stated in \cite{hps:paq} as Conjecture~4.2.   

\begin{thm}\label{thm:ps} 
Fix $k\geq 1$.  Then
$$Q_n(\Pi_{k,2})=Q_n(\Pi_{k,1}) = \sum_{\lambda\in H_{n,k+1}} f^{\overline\lambda}s_\lambda.$$
where $\overline\lambda$ is $\lambda$ with $\lambda_1$ replaced by $\min\{\lambda_1, k\}$.  
\end{thm}

The idea behind our proof is to first establish a descent-preserving bijection between the sets $\fS_n(\Pi_{k,1})$ and 
$\fS_n(\Pi_{k,2})$.  This then reduces the problem to proving that $Q_n(\Pi_{k,2})$ has the desired Schur function decomposition.   By way of Theorem~\ref{thm:hook expansion} this becomes a straightforward calculation.

To construct our bijection we begin with some definitions.  We represent permutations $\si$ via their {\em permutation diagrams} which consist of the points $(i,\si_i)$ in the first quadrant of the plane.
For any $\sigma\in \fS_n$ we say $\sigma_j$ is a \emph{ $k$-endpoint} provided that there exists an occurrence $\sigma_{i_1}\ldots\sigma_{i_k}$ of $\iota_{k}$  in $\sigma$ where $\sigma_{i_k} = \sigma_j$.    Now let $\sigma'$ be the subsequence of $\sigma$ consisting of all the $k$-endpoints. Observe that this  subsequence $\sigma'$ has the nice property that if $\sigma_i$ is in $\sigma'$ then all elements of $\sigma$ that are northeast of $\sigma_i$ are also in $\sigma'$.    Define $r(\sigma)$ to be the left-to-right minima of $\sigma'$ and define $t(\sigma)$ to be the subsequence of $\sigma'$ whose elements are not in $r(\sigma)$. Equivalently $t(\sigma)$ is  the subsequence of all elements in $\sigma$ that are strictly northeast of some element in $r(\sigma)$.  
As such the elements in $t(\sigma)$ are precisely those that are $(k+1)$-endpoints in $\sigma$.  

For a permutation $\sigma$ and subsequence $\tau$ of $\sigma$ we define $\xi=f(\sigma, \tau)$ to be the sequence obtained by replacing the elements in $\sigma$ that occur in $\tau$ with the symbol $\infty$.  We call $i$ a \emph{descent} of $\xi$ provided that $\infty\geq \xi_i> \xi_{i+1}$ and, as usual, denote the set of descents of $\xi$ by $\Des(\xi)$.   Define 
$$\fR_n = \{f(\sigma, t(\sigma)) \ |\ \sigma\in \fS_n\}.$$
For example let $\sigma = 7415326$ and $\tau = 745$.  Then $\xi = \infty\infty 1\infty 326$ with $\Des(\xi) = \{2,4,5\}$.  

In the next two lemmas we show that the permutations in $\fS_n(\Pi_{k,2})$ and those in $\fS_n(\Pi_{k,1})$ have similar decompositions.  In particular we show in Lemma~\ref{lem:char k} that any $\sigma\in \fS(\Pi_{k,2})$ decomposes similarly to the following diagram:

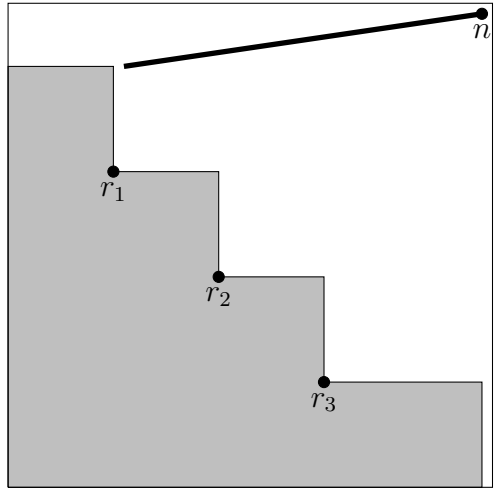
\begin{center}
\begin{tikzpicture}[scale=.7,baseline=(current bounding box.center)]
\draw[fill = lightgray](0,0) -- (0,8) -- (2,8) --(2,6) -- (4,6) -- (4,4)--(6,4)--(6,2)--(9,2)--(9,0)--(0,0);
\draw[line width = 2pt](2.2, 8) -- (9, 9);

\filldraw (2,6) circle (3pt) node[align=left,   below] {$r_1$};
\filldraw (4,4) circle (3pt) node[align=left,   below] {$r_2$};
\filldraw (6,2) circle (3pt) node[align=left,   below] {$r_3$};
\filldraw (9,9) circle (3pt) node[align=left,   below] {$n$};

\draw (0,0) -- (0,9.2) -- (9.2,9.2) -- (9.2,0) -- (0,0);

\end{tikzpicture}.
\end{center}
Here $r(\sigma) = r_1r_2r_3$, the region in gray represents some $\iota_k$-avoiding permutation and white represents regions containing no elements of $\sigma$.  The elements in $t(\sigma)$ form an increasing sequence whose set of values is an interval containing $n$.  This is depicted by the line segment ending at $n$.   Likewise we show in Lemma~\ref{lem:char k-1} that any $\sigma\in \fS_n(\Pi_{k,1})$ decompose similarly to the following diagram:

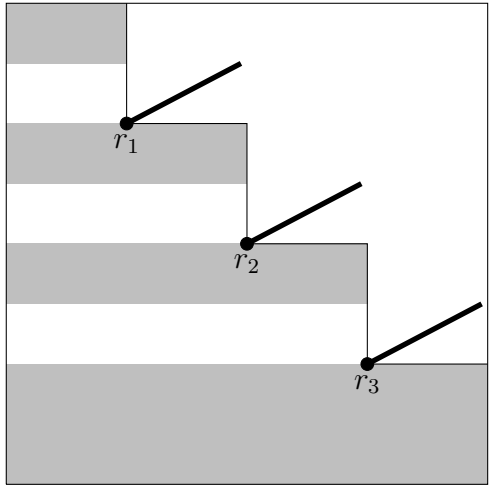
\begin{center}
\begin{tikzpicture}[scale=.8,baseline=(current bounding box.center)]

\fill[lightgray] (0,0) rectangle (8,2);
\fill[lightgray] (0,3) rectangle (6,4);
\fill[lightgray] (0,5) rectangle (4,6);
\fill[lightgray] (0,7) rectangle (2,8);

\draw(2,8) --(2,6) -- (4,6) -- (4,4)--(6,4)--(6,2)--(8,2);

\draw[line width = 2pt](2, 6) -- (3.9, 7);
\draw[line width = 2pt](4, 4) -- (5.9, 5);
\draw[line width = 2pt](6, 2) -- (7.9, 3);

\filldraw (2,6) circle (3pt) node[align=left,   below] {$r_1$};
\filldraw (4,4) circle (3pt) node[align=left,   below] {$r_2$};
\filldraw (6,2) circle (3pt) node[align=left,   below] {$r_3$};

\draw (0,0) -- (0,8) -- (8,8) -- (8,0) -- (0,0);

\end{tikzpicture}
\end{center}
Again, $r(\sigma) = r_1 r_2 r_3$, the gray regions represent some $\iota_k$-avoiding permutations, and white represents regions containing no elements from $\sigma$.  The diagonal lines associated to each $r_i$ represent the elements in $t(\sigma)$ between $r_i$ and $r_{i+1}$ which are increasing and whose values form an interval of the form $(r_i, a_i]$ for some $a_i$.

\begin{lem}\label{lem:char k}
	We have $\sigma\in \fS_n(\Pi_{k,2})$ if and only if $t(\sigma)$  is an increasing sequence whose values form an interval of the form $[a,n]$ for some value $a$.  
Moreover, the function
	$$\varphi: \fS_n(\Pi_{2,k}) \to \fR_n$$
	given by $\sigma \mapsto f(\sigma, t(\sigma))$ is a bijection satisfying $\Des(\sigma) = \Des(\varphi(\sigma))$.
\end{lem}

\begin{proof}
	To prove the forward direction of the first statement, it suffices to show that for any  $b\in t(\sigma)$, all the elements $c>b$ occur  to the right of $b$.  But there is an increasing subsequence $\io$ of $\sigma$ of length $k+1$ ending at $b$.  So if there is a $c>b$ to the left of $b$ then $\sigma$ would contain the subsequence $\io\cup c$ which is a  copy of an element of $\Pi_{k,2}$  which cannot happen.  For the reverse direction let $\sigma$ satisfy the increasing condition.  Suppose, towards a contradiction, that $\sigma$ contains a copy $\ka$ of some element of $\Pi_{k,2}$ and let $b$ and $c$ be the elements of $\ka$ representing $k+1$ and $k+2$, respectively.  Then $b$ is a $(k+1)$-endpoint and so is in $t(\sigma)$.  But now $c>b$ occurs to the left of $b$, contradicting the increasing hypothesis.

	The fact that $\varphi$ is a bijection follows because we can construct its inverse using the description of the elements of 
$\fS_n(\Pi_{k,2})$ just proved.  So given $\tau\in \fR_n$ we merely replace its infinities with the elements of $[a,n]$ in increasing order where $n-a+1$ is the number of infinities.  The description of  $\fS_n(\Pi_{k,2})$ also shows that $i\in\Des(\sigma)$ for  $\sigma\in \fS_n(\Pi_{k,2})$ if and only if either $\sigma_i>\sigma_{i+1}$ and both are not in $t(\sigma)$, or 
$\sigma_i\in t(\sigma)$ and $\sigma_{i+1}\not\in t(\sigma)$.  It easy to see that similar conditions characterize 
$\Des(\varphi(\sigma))$ proving the final statement of the theorem.   
\end{proof}
   
  For our next lemma we need to refine the subsequence $t(\sigma)$.  In particular, for any  $\sigma\in \fS_n$ denote 
$r(\sigma) = r_1r_2\ldots r_s$.  For each $i< s$ let $t_i(\sigma)$ be the subsequence of elements in $t(\sigma)$ that lie in positions between $r_i$ and $r_{i+1}$.  And define $t_s(\sigma)$ to be the subsequence of $t(\sigma)$ that are to the right of $r_s$.  Consequently
  $$t(\sigma) = t_1(\sigma)\cdot t_2(\sigma)\cdot\ \cdots\ \cdot t_s(\sigma),$$  
  where $\cdot$ is concatenation.  

\begin{lem}\label{lem:char k-1}
	We have $\sigma\in \fS_n(\Pi_{k,1})$ if and only if for each $i\leq s$ the sequence $t_i(\sigma)$ is increasing and its values form an interval 
$(r_i,a_i]$ for some $a_i$.  Moreover, the function
	$$\psi: \fS_n(\Pi_{k,1}) \to \fR_n$$
	given by $\sigma \mapsto f(\sigma, t(\sigma))$ is a bijection satisfying $\Des(\sigma) = \Des(\psi(\sigma))$.
\end{lem}

\begin{proof}
	To prove the forward direction of the first statement, it suffices to show that for any $b\in t_i(\sigma)$ all the elements $c\in (r_i,b)$ occur in positions between $r_i$ and $b$.  But there is an increasing subsequence $\io$ of $\sigma$ of length $k+1$ that ends with the pair $r_i b$.   So if there is such a $c$ either before $r_i$ or after $b$, then $\sigma$ would contain the subsequence $\io\cup c$ which is a  copy of an element of $\Pi_{k,1}$.  The reverse direction is also by contradiction.  Suppose that $\sigma$ contains a copy $\ka$ of some element of $\Pi_{k,1}$ and let $b$, $c$, and $d$ be the elements of $\ka$ representing $k+2$, $k+1$ and $k$, respectively.  Then $d$ is a $k$-endpoint and so must be in $[r_i,a_i]$ for some $i$.  This forces $b\in(r_i,a_i]$ since $b$ is bigger than and to the right of $d$.  But then $c\in(r_i,a_i)$ is in a position either before $r_i$ or after $a_i$ which contradicts the increasing condition in this lemma.

As before, we  show that $\psi$ is bijective by constructing its inverse.  Given $\tau\in \fR_n$ we replace the infinities between $r_i$ and $r_{i+1}$ with the elements of $(r_i,a_i]$ in increasing order where $a_i$ is chosen to give the appropriate number of values for the infinity slots.  Lastly the proof that $\psi$ preserves descents is similar to the proof given for the previous lemma and so we omit the details.  
\end{proof}

\begin{proof}[Proof of Theorem~\ref{thm:ps}]
	Combining Lemmas~\ref{lem:char k} and \ref{lem:char k-1} we obtain a descent-preserving bijection between the sets $Q_n(\Pi_{k,2})$ and $Q_n(\Pi_{k,1})$.  As such it now suffices to prove that $Q_n(\Pi_{k,2})$ has the desired Schur decomposition.
	
	Observe that $\Pi_{k,2}=K(R)$ where $R$ is the row superstandard tableau of shape $(k+1, 1)$.  By Theorem~\ref{thm:hook expansion} we then see that
	$$Q_n(\Pi_{k,2}) = \sum_{\lambda} c_\lambda s_\lambda$$
	where $c_\lambda$ is the number of $P\in \Av_n(k+1,2)$ of shape $\lambda$.  It follows from the definitions that  we have  $P\in \Av_n(k+1,2)$ if and only if $\{p_{1,k+1},p_{1,k+2},\ldots, p_{1,n}\} = [p_{1,k+1}, n]$.
It follows that $(2,k+1)\not\in\la$ since we would have to have $p_{2,k+2}>p_{1,k+2}$ and all the elements greater than $p_{1,k+2}$ are to its right.  Thus $\la\in H_{n,k+1}$.  Furthermore, there is a bijection between such tableaux and those of shape $\lab$ gotten by removing the elements $[p_{1,k+1}, n]$.  Thus $c_\la = f^{\lab}$  finishing the proof.
\end{proof}

We conclude this section with a remark.  In our investigation of Conjecture~4.2 we observed that Theorem~\ref{thm:ps} can be generalized. In particular, pick $a,k$ with $k\geq 1$ and $a\leq k+2$ and define 
$$\Pi(a) = (1,2,\ldots, \widehat{a}, \ldots, k,k+1,k+2)\cshuffle a.$$
Then we have $$Q_n(\Pi(a)) = \sum_{\lambda\in H_{n,k+1}} f^{\overline\lambda}s_\lambda.$$ 
To prove this one constructs a descent-preserving bijection, similar to the one given by the above lemmas between the sets 
$\fS_n(\Pi(a))$ and $\fS_n(\Pi(a-1))$.  We choose to omit this details because they are messy and consequently not state this result as a theorem.  Instead we opt to give an example which we hope motivates this more general bijection for the reader.  To this end let $k = 5$ and $a = 5$ so that $\Pi(a) = 123467 \cshuffle 5$.   Now let $\sigma$ be the following $\Pi(5)$-avoiding permutation
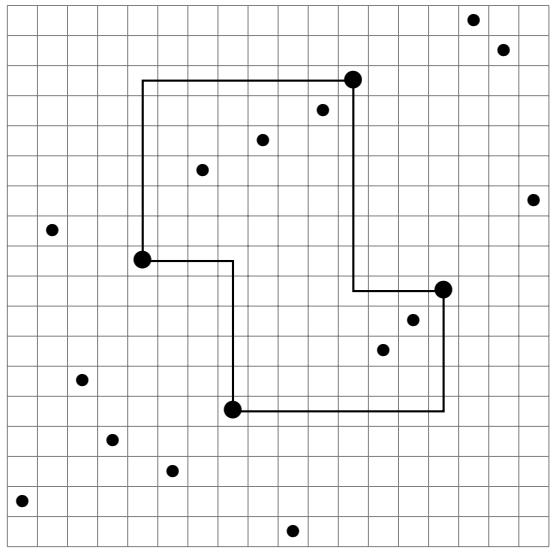
\begin{center}
	\begin{tikzpicture}[scale=.4]
		\draw[step=1cm,gray,very thin] (0,0) grid (18,18);
		\foreach \x/\y in {1/2,2/11,3/6,4/4,5/10,6/3,7/13,8/5,9/14,10/1,11/15,12/16,13/7,14/8,15/9,16/18,17/17,18/12}
		{
			\node at (\x-0.5,\y-0.5) {$\bullet$};
		}	
		
		\node at (5-0.5,10-0.5) {\Large $\bullet$};
		\node at (8-0.5,5-0.5) {\Large $\bullet$};
		
		\node at (15-0.5,9-0.5) {\Large $\bullet$};
		\node at (12-0.5,16-0.5) {\Large $\bullet$};
		
		\draw[thick] (4.5,9.5) -- (4.5,15.5)--(11.5, 15.5)--(11.5,8.5)--(14.5,8.5)--(14.5, 4.5)--(7.5,4.5)--(7.5, 9.5)--(4.5,9.5);

	\end{tikzpicture}
\end{center}
where  $\sigma_5=10$ and $\sigma_8= 5$ are the left-to-right minima among all $(a-2)$-endpoints. (This corresponds to our decomposition of $\Pi_{k,2}$-avoiding permutations as in that case $a = k+2$ and we consider $k$-endpoints.)  Similarly $\sigma_{15} = 9$ and $\sigma_{12}= 16$ are the right-to-left maxima among all points that represent a 1 in an occurrence of $\iota_{k+2-a}$. (In the $\Pi_{k,2}$ case $a = k+2$ so these additional points are ignored.)   These four points are depicted with a larger bullet and the region between them is outlined.  

To motivate how to transform the above picture into a $\Pi(4)$-avoiding permutation recall that in the bijection above we essentially took some $\sigma\in\fS(\Pi_{k,2})$ and let the points in $t(\sigma)$ fall, as if by gravity, ``monotonically" as far as possible while staying above the $r_i$'s.  Treating the points in the outlined region similarly we obtain  
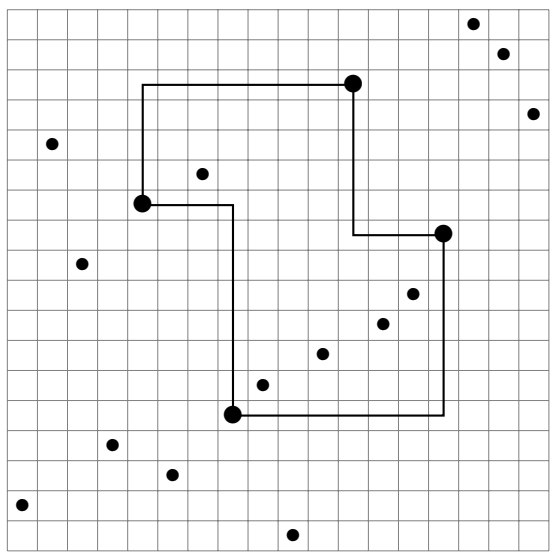
\begin{center}
	\begin{tikzpicture}[scale=.4]
		\draw[step=1cm,gray,very thin] (0,0) grid (18,18);
		\foreach \x/\y in {1/2,2/14,3/10,4/4,5/12,6/3,7/13,8/5,9/6,10/1,11/7,12/16,13/8,14/9,15/11,16/18,17/17,18/15}
		{
			\node at (\x-0.5,\y-0.5) {$\bullet$};
		}	
		
		\node at (5-0.5,12-0.5) {\Large $\bullet$};
		\node at (8-0.5,5-0.5) {\Large $\bullet$};
		
		\node at (15-0.5,11-0.5) {\Large $\bullet$};
		\node at (12-0.5,16-0.5) {\Large $\bullet$};
		
		\draw[thick] (4.5,11.5) -- (4.5,15.5)--(11.5, 15.5)--(11.5,10.5)--(14.5,10.5)--(14.5, 4.5)--(7.5,4.5)--(7.5, 11.5)--(4.5,11.5);

	\end{tikzpicture}
\end{center}
One can check that this permutation is in $\Pi(4)=123567 \cshuffle 4$ as needed.

\section{Pattern-Knuth closed classes}
\label{pkc}

We say $\Pi\sbe \fS$ is \emph{pattern-Knuth closed} if $\fS_k(\Pi)$ is a union of Knuth classes for all $k$.  
Equivalently, $\fSb_k(\Pi)$ is a union of Knuth classes for all $k$.
Note that if $\Pi$ is pattern-Knuth closed then $Q_k(\Pi)$ is Schur nonnegative for all $k$.
This concept was introduced  and studied by Hamaker, Pawloski, and Sagan~\cite{hps:paq}. In this section we continue the investigation of this topic and in doing so answer one of their questions. 

If $\Pi\sbe\fS_k$ is pattern-Knuth closed then, in particular, $\Pi$ must be a union of Knuth classes.  In~\cite{hps:paq} the authors characterized which  $\Pi=K(S)$ for a single SYT $S$ are pattern-Knuth closed.  It turns out that this happens precisely when $S$ is a superstandard hook.  In Theorem~\ref{thm:single} below we give an augmented version of their result and give a more conceptual proof. Our techniques are strong enough that in Theorem~\ref{thm:pairs} we resolve the case where $\Pi$ is a union of two Knuth classes which was left as Question~5.14 in~\cite{hps:paq}.  We also discuss why these results do not seem to generalize to unions of more than two Knuth classes.

We first recall another useful characterization of a Knuth class.  Consider positive integers $a<b<c$.  A {\em Knuth move} in a permutation $\pi$ consists of replacing a factor (adjacent subsequence) of the form $acb$ with one of the form $cab$, or vice-versa.  One is also permitted to exchange factors of the form $bac$ and $bca$.
\bth[\cite{knu:pmg}]
\label{Kequiv}
Two permutations are Knuth equivalent if and only if one can be transformed into the other by a series of Knuth moves.\hqed
\eth

We begin with an elementary property of pattern-Knuth closed sets.

\begin{prop}\label{union}
If $\Pi$ and $\Pi'$ are pattern-Knuth closed, then so is $\Pi \cup \Pi'$.   	
\end{prop}

\begin{proof}
	Observe that $\fS_n(\Pi\cup \Pi') = \fS_n(\Pi)\cap \fS_n(\Pi')$.  As both $\fS_n(\Pi)$ and $\fS_n(\Pi')$ are unions of Knuth classes, it follows that there intersection must be as well.  In other words, $\Pi\cup \Pi'$ is pattern-Knuth closed.
\end{proof}

Let $\xi$ and $\ze$ are (not necessarily disjoint) subsequences of a permutation $\si$. 
Define $\xi\cup \ze$ to be the subsequence of $\si$ whose elements consists of those of $\xi$ together with those of $\ze$.
We write $\xi\uplus\ze$ if $\xi$ and $\ze$ are disjoint.  
The {\em shape} of a permutation, $\sh\si$, is the shape of its output tableaux under the Robinson-Schensted map.  Finally, let 
$\lis\si$ (respectively $\lds\si$) stand for the length of a longest increasing (respectively decreasing) subsequence of $\si$.

\begin{lem}\label{uid}
Suppose that $\si=\io\uplus\de$ where $\io$ is increasing of length $a$ and $\de$ is decreasing of length $b$. Then $\sh\si$ is one of the following
$$
(a,1^b),\ (a+1,1^{b-1}),\ (a,2,1^{b-2}).
$$
\end{lem}
\begin{proof}
Because of the hypothesis we have $\lis\si=a$ or $a+1$ and $\lds\si=b$ or $b+1$.  We have three cases.

If $\lds\si=b+1$ then the first column of $\sh\si$ is of length $b+1$ by Theorem~\ref{RS} (b) and (c).  Furthermore, the length of the first row of $\sh\si$ is at least $a$.  But $\#\si=a+b$ so we must have $\sh\si=(a,1^b)$.  Similarly if $\lis\si=a+1$ then this forces $\sh\si=(a+1,1^{b-1})$.  Finally, assume $\lds\si=b$ and $\lis\si=a$.  So the first column of $\sh\si$ has $b$ elements and the first row has $a$, giving a total of $a+b-1$ entries. Thus the remaining entry must be in the $(2,2)$ box.
\end{proof}

We point out that the third case of this lemma can occur.  In fact, the smallest example where this case is needed is when $\si=65127843$.  Here $\si$ contains a unique increasing sequence of length $a =4$, namely $1278$, and a unique decreasing sequence of length $b= 4$, namely $6543$, and $\sh\si=(4,2,1,1)$.

For any partition $\mu$ recall that 
$$
K(\mu) =\bigcup_{T\in\SYT(\mu)}K(T)
$$
and define
$$
\fS_n(\mu)=\{\la\ptn n \mid \la\not\spe\mu\}.
$$

\bth\label{pkc shape}
Let $\mu = (a,1^b)$ and $\tau= (a,2,1^{b-1})$.  Then $\Pi=K(\mu) \cup K(\tau)$  is pattern-Knuth closed and 
$$
 Q_n(\Pi) = \sum_{\la\in\fS_n(\mu)}f^{\lambda}s_\lambda.
$$
\eth
\begin{proof}
To prove both statements, it suffices to show that for any permutation $\si$ we have $\si\in\fSb_n(\Pi)$ if and only if $\sh\si\spe(a,1^b)$.  For the forward direction, let $\ka$ be a copy of an element of either $K(\mu)$ or $K(\tau)$ in $\si$.  Then 
$\ka$, and hence $\si$, contains an increasing subsequence of length $a$ and a decreasing subsequence of length $b+1$.
It follows that 
$\sh\si\spe(a,1^b)$ as desired.

For the reverse, the assumption on $\sh\si$ means that $\si$ has a subsequence $\ka=\io\cup\de$ where $\io$ is increasing of length $a$ and $\de$ is decreasing of length $b+1$.  If the union is disjoint, then by the previous lemma we must have $\sh\ka$ is one of $(a,1^{b+1})$, $(a+1,1^b)$, or $(a,2,1^{b-1})$.  In the third case we are done since $\si$ contains an element of $K(\tau)$.  If we are in one of the first two cases then one can remove an element of $\de$ or $\io$, respectively, to show that $\si$ contains an element of $K(\mu)$.  If the union is not disjoint then $\io$ and $\de$ must overlap in precisely one element and an argument as in the proof of the lemma shows that $\sh\ka=(1^a,b)$, finishing the proof.
\end{proof}

We now begin our characterization of certain pattern-Knuth closed sets with some critical definitions.  We call the descents of $\pi^{-1}$ the \emph{$i$-descents} of $\pi$ and denote the set of all $i$-descents by $\iDes (\pi)$.  
An equivalent definition of $i$-descents is the following:  $a\in \iDes(\pi)$ if and only if $a+1$ is to the left of $a$ in $\pi$. 
We say a set of permutations $\Pi \sbe\fS_n$ is \emph{$i$-descent consistent} provided that $\iDes(\pi) = \iDes(\sigma)$ for all $\pi,\sigma\in \Pi$ and write $\iDes(\Pi)$ for this common set of $i$-descents.
Equivalently,  $\Pi\sbe\fS_n$ is $i$-descent consistent if and only if $\Pi\sbe D_J^{-1}$ for some $J\sbe[n-1]$, where 
\beq\label{eq:D_J}
D_J=\{\pi\in\fS_n  \mid \Des(\pi) = J\}
\eeq
and we take the inverse of a set by taking the inverse of each of its elements.
Recall that for any $S\in\SYT(n)$ we have $\iDes(\sigma) = \Des(S)$ for each $\sigma\in K(S)$ by Theorem~\ref{RS} (a) and (d).  Furthermore for any $J\sbe[n-1]$ 
$$ \bigcup_{S} K(S)=D_J^{-1}$$
where the union is over all $S\in \SYT(n)$ where $\Des(S) = J$. So, Knuth classes give a natural example of $i$-descent consistent sets. The next lemma will be important in our characterization of the pattern-Knuth closed sets which consist of a single Knuth class.
\begin{lem}\label{lem:ssh}
Fix $S\in \SYT(n)$.  Then the following are equivalent
 \begin{enumerate}[label =  \roman*)]
 	\item[(i)] $K(S)= D_J^{-1}$  for some $J\sbe[n-1]$,
	\item[(ii)] $K(S)= D_J^{-1}$ where $J = [1,k]$ or $[k,n-1]$ for some $1\le k\le n-1$,
 	\item[(iii)] $S$ is a superstandard hook.
 \end{enumerate}
\end{lem}
\bprf
Clearly (ii) implies (i).  The fact that (iii) implies (ii) follows directly from the definition of a superstandard hook.  The proof that (i) implies (iii) is by contradiction.  Assuming that $S$ is either a non-superstandard hook or not of hook shape, it is easy to find another tableau $T$ with $\Des T =\Des S = J$ so that $K(T)\sbe D_J^{-1}$.  Since Knuth classes are disjoint, this implies $K(S)\sbs D_J^{-1}$ which is the desired contradiction.
\eprf

The fact that Knuth classes are $i$-descent consistent gives a criterion for determining  when two permutations $\pi$ and $\sigma$ have distinct insertion tableaux:  If $\iDes(\pi) \neq \iDes(\sigma)$ then $P(\pi) \neq P(\sigma)$. We make repeated use of this criterion in what follows.  

Another related notion needed is that of swap closure.  This concept is due to Joel Lewis~\cite{pc:jl}.   For any $\pi\in \fS_n$ the operation of interchanging adjacent elements $\pi_i$ and $\pi_{i+1}$ in $\pi$ where $|\pi_i-\pi_{i+1}|>1$ is called a \emph{swap}.  We say two permutations are \emph{swap equivalent} if one can be obtained from the other via a sequence of swaps.  A set of permutations is called  \emph{swap closed} if it is closed under this equivalence relation.  In what follows we restrict our attention to swaps involving the largest element $n$.  As such we define $\overrightarrow\pi$ to be the result of swapping $n$ with its right neighbor.  In the case that this neighbor is $n-1$ or $n$ is the rightmost element of $\pi$ we set $\overrightarrow\pi = \pi$.  We also define $\overdoublerightarrow{\pi}$ to be the permutation obtained from $\pi$ by removing $n$ from its position and placing it on the right end of $\pi$.  We define 
$\overleftarrow{\pi}$ and $\overdoubleleftarrow{\pi}$ analogously.

The relationship between swap closure and the sets $D_J^{-1}$ is given by the following lemma. The following result was also obtained by Lewis but not published~\cite{pc:jl}.   

\begin{lem}\label{lem:swpclosed}
Let $\emptyset \neq \Pi\sbe \fS_n$.  Then $\Pi$ is swap closed if and only if 
$$
\Pi=\bigcup_{i=1}^s D_{J_i}^{-1}
$$
for some $J_1,\ldots, J_s\sbe [n-1]$. 
\end{lem}
\begin{proof}
We first claim that if $\emptyset\neq \Pi\sbe D_J^{-1}$ for some $J\sbe[n-1]$ and $\Pi$ is swap closed then $\Pi = D_J^{-1}$.  Let 
$J= \{j_1,\ldots, j_k\}$ and define
	$$\pi_J := j_k+1,\ldots, n, j_{k-1}+1,\ldots j_k,\ldots,  j_1+1,\ldots, j_2, 1,\ldots, j_1\in D_J^{-1}.$$
	  We claim that $\pi_J$ is swap equivalent to each $\sigma\in D_J^{-1}$.  As $\iDes(\sigma) = J$ we have the maximal increasing subsequence $1\ldots j_1$ in $\sigma$.  Since $j_1+1$ is to the left of $j_1$ we can move the elements  $j_1,j_1-1,\dots,1$ in that order to the end of $\sigma$ by a sequence of swaps leaving all the other elements of $\sigma$ in the same relative positions.  Repeating this process yields $\pi_J$.   Thus all the elements  $\si\in D_J^{-1}$ are swap equivalent.  Since $\Pi\neq \emptyset$ and is a swap closed  subset of $D_J^{-1}$ we conclude that $\Pi = D_J^{-1}$.
	
Now assume $\Pi$ is nonempty and swap closed. By an argument similar to the previous paragraph, for each $J\sbe[n-1]$ such that 
$\Pi\cap D_J^{-1}  \neq \emptyset$ we have  $D_J^{-1}\sbe \Pi $.  The forward direction of our lemma now follows.  Since swaps interchange elements that differ by at least 2 we see that $\iDes$ is invariant under swaps.  As each $D_{J_i}^{-1}$ is the set of all permutations with $i$-descent set $J_i$ the reverse direction follows.
\end{proof}

We now wish to make a connection between pattern-Knuth closure and swap closure.  Note that the second statement of this result follows from the first and the previous lemma.

\begin{thm}\label{thm:swpclosed}
	If $\Pi\subseteq \fS_n$ is both pattern-Knuth closed and $i$-descent consistent then $\Pi$ is swap closed.  Furthermore, when $\Pi\neq \emptyset$ we have $\Pi=D_J^{-1}$ for some $J\sbe[n-1]$.  
\end{thm}

To prove this theorem we begin with some preliminaries.   Given a permutation $\pi\in \fS_n$, and integers $1\leq i\leq n$ and $1\leq m\leq n$ we can construct a new permutation $\sigma$ by standardizing 
$$\pi_1,\ldots, \pi_{i-1},m^{+}, \pi_{i},\ldots, \pi_n$$
where $m^+:=m+1/2$.   For simplicity we refer to this operation by saying that $\sigma$ is the result of \emph{adding $m^+$ to $\pi$ in position $i$}.   Of course we may also add $m^-= m-1/2$ to $\sigma$ where this is defined analogously.  When adding an element to a permutation we take the standardization to be implicit.  For example if we add $3^-$ to $\pi = 132$ in position 3 we write $133^-2\in \fS_4$ instead of $1432\in \fS_4$ and refer to $3^-$ or $3$ instead of $3$ or $4$ respectively. 

We  write $\pi - m$ to denote the permutation obtained by deleting the value $m$ from $\pi$.  
When subtracting an element, we always refer to the elements of $\pi-m$ in their standardized form, as opposed to the convention for adding an element.  For example, if $\pi\in\fS_n$ then in $\pi-(n-1)$ the element $n-1$ is where $n$ is in $\pi$.
When $m$ is the largest value we instead write $\widehat\pi$ for $\pi-n$ and extend this notation to sets of permutations in the usual way.  Finally we also use this notation in the context of standard Young tableaux.  For any $S\in \SYT(n)$ we define $\widehat S$ to be the standard Young tableau obtained by deleting $n$ from $S$.

\begin{lem}\label{4points}
Assume $\pi\in \fS$ contain the subsequence $m-3, m-1, m, m-2$ for some $m$.  Then $\iDes(\pi - x) = \iDes(\pi - (m-2))$ if and only if $x = m-2$.  
\end{lem}
\begin{proof}
Because of the given configuration of elements it follows that we have $m-3,m-2\notin \iDes(\pi-(m-2))$.  
If $x< m-2$ then we see  $m-3\in\iDes(\pi-x)$ since $m-1$ is to the left of $m-2$ in $\pi$.  Similarly, if $x>m-2$ then $m-2\in \iDes(\pi-x)$ since both $m-1$ and $m$ are to the left of $m-2$ in $\pi$.  The lemma now follows.  
\end{proof}

\begin{lem}\label{swpright}
Let $\Pi\subseteq \fS_n$ be pattern-Knuth closed.  Assume $\Pi$ is such that $n-1\notin\iDes(\sigma)$ for all $\sigma\in \Pi$.  Then $\overrightarrow\pi\in \Pi$ for all  $\pi\in\Pi$.  
\end{lem}
\begin{proof}
	Fix $\pi\in \Pi$ with $\pi_i = n$ so that $\pi_{i+1}<n$.   Now add $n^-$ to $\pi$ in position $i+2$ and use it to interchange $n$ and $\pi_{i+1}$ via a Knuth move to obtain
	$$\rho:=\ldots, \pi_{i+1},n, n^-,\ldots\in \fSb_{n+1}(\Pi).$$
	Since all patterns in $\Pi$ have length $n$, there exists some $x$ such that $\rho-x\in \Pi$. If $x\neq n, n^-$ then $n-1\in \iDes(\rho-x)$ in which case $\rho-x\notin \Pi$.  So $x$ is one of $n,n^-$ and  $\overrightarrow\pi = \rho-x\in \Pi$.  
\end{proof}

Next we prove a lemma to help with our inductive proofs of both Theorems~\ref{thm:swpclosed} and \ref{thm:pairs}.

\begin{lem}\label{pkcinduct}
Assume $\Pi\sbe \fS_n$ is pattern-Knuth closed with the property that for each $\pi\in \Pi$, we have either $\overdoublerightarrow{\pi}\in \Pi$ or $\overdoubleleftarrow{\pi}\in \Pi$.
Then $\widehat{\Pi}$  is pattern-Knuth closed.  
\end{lem}
\begin{proof}
To show $\widehat\Pi$ is pattern-Knuth closed, take $\sigma\in \fSb_k(\widehat\Pi)$ for some $k$ and consider any $\rho\in \fS_k$ Knuth-equivalent to $\sigma$ with the aim of showing $\rho\in \fSb_k(\widehat\Pi)$.    
In particular, assume $\sigma$ contains the pattern $\widehat\pi\in \widehat{\Pi}$.  Now consider the case when $\overdoublerightarrow{\pi}\in \Pi$ and observe that the concatenation $\sigma,k+1$ contains $\overdoublerightarrow{\pi}$ and hence $\sigma,k+1\in \fSb_{k+1}(\Pi)$.    
Since $\rho$ and $\sigma$ are Knuth-equivalent we have $\rho,k+1$ and  $\sigma,k+1$ are too.  
As $\Pi$ is pattern-Knuth closed $\rho, k+1\in \fSb_{k+1}(\Pi)$. Hence $\rho\in \fSb_{k}(\widehat\Pi)$ as needed.  

The case when  $\overdoubleleftarrow\pi\in \Pi$ follows by an analogous argument and so the details are omitted. 
\end{proof}

We are now in a position to prove Theorem~\ref{thm:swpclosed}.

\begin{proof}[Proof of Theorem~\ref{thm:swpclosed}]

The second assertion follows from the first and Lemma~\ref{lem:swpclosed}. The first statement of the theorem certainly holds when $n=1$.  We now proceed by induction on $n$ with $\Pi\sbe\fS_n$ where $n>1$.  By considering $\Pi^r$ if necessary we may assume $n-1\notin \iDes(\Pi)$.

It follows from Lemma~\ref{swpright} that $\overrightarrow{\pi} \in \Pi$ for all $\pi\in \Pi$.
So we can always swap $n$ to the right and remain in $\Pi$. Hence $\overdoublerightarrow\pi\in \Pi$ for all $\pi\in \Pi$ since $n-1$ is to the left of $n$ in $\pi$.  Therefore we know by Lemma~\ref{pkcinduct} that $\widehat\Pi\sbe \fS_{n-1}$ is pattern-Knuth closed.  Clearly $\widehat\Pi$ is $i$-descent consistent and so we conclude by induction that $\widehat\Pi$ is swap closed. This means that if we take any $\overdoublerightarrow{\pi}\in \Pi$ and swap elements neither of which are $n$ then the result is in $\Pi$. Consequently it now suffices to show that $\overleftarrow{\pi}\in \Pi$ for any  $\pi\in \Pi$.  To this end fix $\pi\in \Pi$ and set $\pi_i=n$ where $\pi_{i-1}\le n-2$.  We consider two cases.

\medskip
\noindent
\textbf{Case 1:} $n-2\notin \iDes(\Pi)$

In this case $n-2, n-1, n$ is a subsequence of $\pi$.    Add $(n-1)^-$ to $\pi$ in position $i+1$ and use it to interchange $\pi_{i-1}$ and $n$ via a Knuth move to obtain
$$
\rho:=\ldots, n-2,\ldots, n-1,\ldots, n,\pi_{i-1},(n-1)^-,\ldots\in \fSb_{n+1}(\Pi).
$$    
Let $x$ be such that $\rho-x\in \Pi$ so that $\iDes(\rho-x) = \iDes(\Pi)$.  Observe that $\rho-(n-1)^- = \overleftarrow{\pi}$ and hence $\iDes(\rho-(n-1)^-) = \iDes(\Pi)$.  By Lemma~\ref{4points} it now follows that we must have $x= (n-1)^-$ proving, in this case, that $\overleftarrow{\pi}\in \Pi$.  

\medskip
\noindent
\textbf{Case 2:} $n-2\in\iDes(\Pi)$

Add $n^-$ to $\pi$ in position $i-1$ and use it to interchange $n$ and $\pi_{i-1}$ via a Knuth move to obtain 
	$$\rho:=\ldots, n-1,\ldots, n^-,n,\pi_{i-1},\ldots \in \fSb_{n+1}(\Pi),$$ 
	where $n-2$ is not shown but is to the right of $n-1$.  Let $x$ be such that $\rho-x\in \Pi$ and $\iDes(\rho-x) = \iDes(\Pi)$.  As $n-2\in\iDes(\Pi)$ it follows that $x = n-1, n^-$, or $n$.  If $x= n, n^-$ we are done as $\overleftarrow{\pi} = \rho-x \in \Pi$.

	If $x = n-1$ then $n-2$ must be to the right of $n$ in $\rho$ so that  $n-2\in\iDes(\rho-x)$. As $n$ and $n-2$ are to the right of $n^-$ in  $\rho$ it follows that $\rho-n^-$ is obtainable from  $\rho-(n-1)$ by swapping $n-1$ left.  Define $\sigma=\rho -(n-1)$ so 
	$$\sigma= \ldots, \pi_{i-2},n-1,n,\pi_{i-1},\ldots\in \Pi$$
	where $n-2$ is not shown but is to the right of $n$.  It now suffices to show that we can swap $n-1$ to the left an arbitrary number of times and stay in $\Pi$.  Add $(n-1)^-$ to $\sigma$ in position $i-2$ and apply a Knuth move to obtain
	$$\ldots, (n-1)^-,n-1,\pi_{i-2},n,\pi_{i-1},\ldots \in \fSb_{n+1}(\Pi).$$
As $n-2\in \iDes(\Pi)$, we must delete $y = (n-1)^-,n-1$, or $n$ to obtain a pattern in $\Pi$.  If $y$ is one of the first two we are done.  If $y = n$, then, by the second paragraph in this proof, we can swap $n$ to the right once so that we obtain $\si$ with $n-1$ and 
$\pi_{i-2}$ interchanged which must still be in $\Pi$.  Repeating this argument demonstrates that we can swap $n-1$ left as needed.  
	\end{proof}

In \cite{hps:paq} the authors prove in Theorem~5.8 that $K(T)$ is pattern-Knuth closed if and only if $T$ is a superstandard hook.  We are now in position to give a more conceptual explanation as to why this theorem holds as well as  place its statement in a more general framework.  
\begin{thm}\label{thm:single}
Suppose $\Pi= K(S)$ for some $S\in \SYT(n)$.  Then the following are equivalent:
\begin{enumerate}[label = \roman*)]
	\item[(i)] $\Pi$ is pattern-Knuth closed,
	\item[(ii)] $\Pi$ is swap closed,
	\item[(iii)] $\Pi = D_J^{-1}$ where $J = [1,k]$ or $J = [k,n-1]$ for some $k$,  
	\item[(iv)] $S$ is a superstandard hook.
\end{enumerate}
\end{thm}

\begin{proof}
Our definition of $\Pi$ implies that $\Pi$ is $i$-descent consistent. The implication (i)  implies (ii)) follows directly from Theorem~\ref{thm:swpclosed}. The fact that (ii) implies (iii)  follows from the fact that $\Pi$ is $i$-descent consistent as well as 
Lemmas~\ref{lem:ssh} and \ref{lem:swpclosed}.  The implication (iii) implies (iv) also follows from Lemma~\ref{lem:ssh}.   Lastly, the fact that (iv) implies (i) is given a straightforward explanation in the paper of~\cite{hps:paq}. 
\end{proof}

We now characterize pattern-Knuth closed classes that are unions of two Knuth classes, answering a question in~\cite{hps:paq}.
We say $S\neq T\in \SYT(n)$ are an \emph{$i$-descent-complete} pair if $K(S)\cup K(T) = D_J^{-1}$ for some $J\sbe[n-1]$. We start by characterizing such pairs.

\begin{lem}\label{lem:doubles}
Suppose $S\neq T\in \SYT(n)$.  Then $S$ and $T$ are an $i$-descent-complete pair if and only if

\beq\label{eq:pair1}
\ytableausetup{boxsize=1.5em}
S = \raisebox{7mm}{\begin{ytableau}
\scriptstyle 1&\scriptstyle 2& \scriptstyle \cdots & \scriptstyle k&\scriptstyle n\\
\scriptstyle k{+}1\\
\scriptstyle \vdots\\
\scriptstyle n{-}1
\end{ytableau}}\quad \qquad
T = \raisebox{7mm}{\begin{ytableau}
\scriptstyle 1&\scriptstyle 2& \scriptstyle \cdots & \scriptstyle k\\
\scriptstyle k{+}1& \scriptstyle n\\
\scriptstyle \vdots\\
\scriptstyle n{-}1
\end{ytableau}} 
\eeq
where $2\leq k\leq n-2$ and $\Des(S) = \Des(T) = [k,n-2]$, or
\beq\label{eq:pair2}
S = \raisebox{7mm}{\begin{ytableau}
\scriptstyle 1&\scriptstyle 2& \scriptstyle k & \scriptstyle \cdots &\scriptstyle n\\
\scriptstyle 3\\
\scriptstyle \vdots\\
\scriptstyle k-1
\end{ytableau}}\qquad \qquad
T = \raisebox{7mm}{\begin{ytableau}
\scriptstyle 1&\scriptstyle 2& \scriptstyle k+1 & \scriptstyle \cdots & \scriptstyle n\\
\scriptstyle 3 &\scriptstyle k\\
\scriptstyle \vdots\\
\scriptstyle k-1
\end{ytableau}} 
\eeq 
where $4\le k\le n$ and $\Des(S) = \Des(T) = [2,k-2]$, or $S$ and $T$ are transposes of these tableaux.  

\end{lem}

\begin{proof}
The reverse direction follows from a straightforward check.  To prove the forward direction let $J\sbe[n-1]$ be such that $K(S)\cup K(T) = D_J^{-1}$.  As $S\neq T$ it follows from Lemma~\ref{lem:ssh} that neither $S$ nor $T$ can be superstandard.  Furthermore, by considering reverses if necessary we may assume $1\not\in J$  so that $2$ is in the first row of both $S$ and $T$.  It  now suffices to show that no other pairs of tableaux other than those  displayed in (\ref{eq:pair1}) or (\ref{eq:pair2}) are $i$-descent-complete pairs.  

To this end observe that the mapping between hook tableaux and subsets of $[n-1]$ given by $\Des$ is a bijection.  As $K(S)\cup K(T)$ is the set of all permutations with $i$-descent set $J$ it follows, that either $S$ or $T$ is hook shape.  We take $S$ to be of hook shape. In particular as $S$ is not superstandard its leftmost column has length at least 2.      Now consider adjacent elements $a$ and $b$ with $2\leq a<b$ in the top row of $S$.  If $b>a+1$ then $b-1$ must be in the first column of $S$ so that the tableau obtained by moving 
$b$ into position $(2,2)$ is standard and has descent set $J$.  Likewise consider adjacent elements $c$ and $d$ with $2\leq c<d$ in the first column of $S$.  If $d>c+1$ then the tableau obtained by moving $d$ into position $(2,2)$ is standard and also has descent set $J$.  As there are exactly two tableaux with descent set $J$ we must have
$$S = \raisebox{7mm}{\begin{ytableau}
\scriptstyle 1&\scriptstyle 2& \scriptstyle \cdots &\scriptstyle a& \scriptstyle b&\scriptstyle\cdots &\scriptstyle n\\
\scriptstyle a+1\\
\scriptstyle a+2\\
\scriptstyle \vdots\\
\scriptstyle b{-}1
\end{ytableau}},$$
for some $a\ge2$ and $a+1< b\leq n$.  
If $a=2$, then this results in the second pair in the statement of the theorem.  If $a\ge3$ then
we claim that  $b=n$. For if $b\le n-1$ then we have at least two additional tableaux with descent set $J$.  Namely, we have the tableau $S'$ obtained by moving $b$ into position $(2,2)$ and the tableau $S''$ obtained from $S'$ by moving $b+1$ into positions $(2,3)$ where the fact that $a\ge3$ guarantees that $S''$ is standard.  When $b=n$ we get the first pair of tableaux in the statement of the theorem. This completes our proof. 
\end{proof}

We are now ready to state our second main theorem of this section.

\begin{thm}\label{thm:pairs}
	Suppose  $\Pi =  K(S)\cup K(T)$  where $S\neq  T\in\SYT(n)$. The following are equivalent:
\begin{enumerate}[label = \roman*)]
	\item[(i)] $\Pi$ is pattern-Knuth closed,
	\item[(ii)] $\Pi$ is swap closed,
	\item[(iii)] $\Pi = D_J^{-1}\cup D_L^{-1}$ where either $J\neq L$ are of the form given in Lemma~\ref{lem:ssh}, or $J=L$ is of the form given in Lemma~\ref{lem:doubles},
	\item[(iv)] $S$ and $T$ are either distinct superstandard hooks, the tableaux pairs displayed in Lemma~\ref{lem:doubles}, or their transposes.
\end{enumerate}
\end{thm}

In light of Theorems~\ref{thm:single} and \ref{thm:pairs} one would hope that pattern-Knuth closure would, in general, be equivalent to swap closure.  Unfortunately, this is not true.  For example take $\Pi = K(T_1)\cup K(T_2) \cup K(T_3)$ where 
$$\ytableausetup{boxsize=1.3em}
T_1= \begin{ytableau}
\scriptstyle 1&\scriptstyle 2& \scriptstyle 4\\
\scriptstyle 3\\
\end{ytableau}\quad 
T_2 = \begin{ytableau}
\scriptstyle 1&\scriptstyle 3& \scriptstyle 4\\
\scriptstyle 2\\
\end{ytableau}\quad 
T_3 = \begin{ytableau}
\scriptstyle 1&\scriptstyle 2& \scriptstyle 3\\
\scriptstyle 4\\
\end{ytableau}\ .$$
A computer check shows that $\fS_5(\Pi)$ is a union of Knuth classes and so, by Lemma~\ref{pkc:M} below, we know that this $\Pi$ is pattern-Knuth closed.    On the other hand this $\Pi$ is not swap closed as $3124\in K(T_1)$ but performing a swap gives 
$3142\in K(T_4)$ where 
$$T_4 = \ytableausetup{boxsize=1.3em}\begin{ytableau}
\scriptstyle 1&\scriptstyle 2\\
\scriptstyle 3& \scriptstyle 4\\
\end{ytableau}\ . $$

We now turn to the proof Theorem~\ref{thm:pairs}. The crux of this proof is in demonstrating that i) implies ii).  We begin by building up the required lemmas to show this implication.    In what follows we denote Knuth equivalence by $\sim$.

\begin{lem}\label{lem:RS and ds}
Fix $S\in \SYT(n)$.  The following are equivalent:
\begin{enumerate}[label=\roman*)]
	\item[(i)] there exists some $\pi\in K(S)$ with $\pi_n = n$,
	\item[(ii)] $K(\widehat S) = \widehat{K(S)}$ and $n$ is in the top row of $S$,
	\item[(iii)] $\overdoublerightarrow{\sigma} \in K(S)$ for all $\sigma\in K(S)$.
\end{enumerate}
\end{lem}

\begin{proof}
Before proving any of the above implications observe that if we apply the Robinson-Schensted algorithm to any $\sigma\in\fS_n$ and keep track of only the values $<n$ it is clear that $P(\widehat\sigma)=\widehat{P(\sigma)}$. Thus in general $\widehat{K(S)} \sbe K(\widehat S)$.

We now prove (i) implies (ii).  By the existence of $\pi\in K(S)$ with $\pi_n=n$ it is clear that $S=P(\pi)$ has $n$ in the top row.  From the observation in the first paragraph, to establish the equality  we need only show that $K(\widehat S) \sbe \widehat{K(S)}$ and we know $P(\widehat\pi) = \widehat{P(\pi)} = \widehat{S}$.  Now pick $\sigma\in K(\widehat S)$ so that $\sigma\sim \widehat\pi$ and hence $\sigma,n\sim \pi$.  Consequently,  $P(\sigma,n) = P(\pi) = S$.  Therefore $\sigma\in \widehat{K(S)}$ as needed.  

We now prove that (ii) implies (iii).  By (ii) we see that for any $\sigma\in K(S)$ we have $\widehat\sigma\in \widehat{K(S)}= K(\widehat{S})$.  So $P(\overdoublerightarrow{\sigma})$ is obtained by adding $n$ to the end of the top row of $\widehat{S}$.  As $n$ is in the top row of $S$ it follows that $P(\overdoublerightarrow{\sigma}) = S$.  

The fact that (iii) implies (i) is clear.  
\end{proof}

\begin{lem}\label{PairsHooks}
	Fix $n\geq 1$.  Let $S,T\in \SYT(n)$ be of hook shape with the property that $\widehat S$ and $\widehat T$ are superstandard hooks. If $\Pi = K(S)\cup K(T)$ is pattern-Knuth closed then $S$ and $T$ are both superstandard.  
\end{lem}

\begin{proof}
When $n\leq 3$ all tableaux are superstandard and so the result follows trivially.  When $n=4$ a computer check demonstrates the theorem.  Therefore we may assume $n\geq 5$.  	For a contradiction assume $T$ is not superstandard.  By considering $\Pi^r$  if necessary we may further assume $n$ is in the top row of $T$ and that the row word of $T$ is
	$$\rho(T) = n-1,\ldots, k+1, 1,2,\ldots, k, n$$
	where $2\leq k\leq n-2$.  We now consider three cases.

\medskip\noindent
\textbf{Case 1:} $2<k< n-2$

Add $(n-1)^-$ to $\rho(T)$ in position 1 to obtain
	$$(n-1)^-, n-1,\ldots, k+1, 1,2,\ldots, k, n\in \fSb_{n+1}(\Pi).$$
Via Knuth moves slide $n-1$ to position $n-1$ and then, via another Knuth move, interchange $n$ and $k$ to obtain
$$\tau: = (n-1)^-, n-2,\ldots, k+1,1,2,\ldots,k-1, n-1, n, k\in \fSb_{n+1}(\Pi).$$
There exists some $x$ such that $\tau-x\in \Pi$. Because $k<n-2$ the decreasing prefix $(n-1)^-, n-2, \ldots, k+1$ has length at least 2.  Now a straightforward check shows that if $x\neq k$ then $P(\tau-x)$ is not of hook shape since $k$ bumps either $n-1$ or $n$ into position $(2,2)$. When $x =k$ we see that $P= P(\tau-x)$ has hook shape with top row $1,2,\dots,k-1,n-1,n$.  The fact that 
$2< k<n-2$ means that $\widehat P$ is not superstandard and so $P \neq S,T$. This contradiction shows that the lemma holds in this case.

\medskip\noindent
\textbf{Case 2:} $k=n-2$

Here $\rho(T) = n-1, 1,2,\ldots,n-2, n\sim 1,2,\ldots,n-3,n-1,n, n-2$. Now add $(n-2)^+$ in position $n-3$ and use it to interchange $n-3$ and $n-1$ to obtain
$$
\barr{l}
1,2,\ldots,n-4, (n-2)^+, n-1,n-3,n, n-2 \\[5pt]
\qquad\sim 1,2,\ldots,n-4, (n-2)^+, n-1,n,n-3, n-2\in \fSb_{n+1}(\Pi).	
\earr
$$
Denoting the last permutation by $\tau$, there exists $x$ be such that $\tau -x\in \Pi$.  
If $x\neq n-2,n-3$ then the shape of $P=P(\tau-x)$ is $(n-2,2)$ which is not a hook.
If $x=n-2$ or $n-3$ then 
$$
\ytableausetup{boxsize=1.6em, }
P=\begin{ytableau}
\scriptstyle 1 & \scriptstyle 2 & \scriptstyle \ldots &  \scriptstyle n{-}4 & \scriptstyle n{-}3 & \scriptstyle n{-}1 & \scriptstyle n  \\
\scriptstyle n{-}2
 \end{ytableau}\ .
$$
Since $n\geq 5$ we see that the top row contains 1 and 2 and $n-1\geq 4$ but not $n-2$.  So $\widehat P$ is not superstandard and hence $P\neq S,T$.  We conclude, in this case, that the lemma holds.  

\medskip\noindent
\textbf{Case 3:} $k=2$

Here $\rho(T) = n-1,\ldots,3, 1,2,n$.  Now add $(n-2)^+$ in position $n-1$ and use it to interchange $2$ and $n$ to obtain 
$$
\barr{l}
n-1,\ldots,3, 1,(n-2)^+,n,2 \\[5pt]
\qquad\sim n-2,n-1,n,1,(n-2)^+,n-3,\ldots,3,2\in \fSb_{n+1}(\Pi).
\earr
$$
Denoting the last permutation by $\tau$ there must exist some $x$ so that $\tau-x\in \Pi$.   Note that the first $6$ terms of $\tau$ are order isomorphic to $356142$ whose insertion tableau is not of hook shape.  So $x$ must be one of the first $6$ terms.  If $x\neq 1$ then the remaining $5$ elements in $\tau-x$ insert to a tableau which is not of hook shape.  Therefore $x=1$ and 
$\tau -1 = n-3,n-1,n,n-2,n-4,\ldots, 1$.  It is now easy to check that if  $P=P(\tau-1)$  then $\widehat P$ is not superstandard and consequently $P\neq S,T$.  This completes the final case and the proof of the lemma.
\end{proof}

In what follows we denote the symmetric difference of two sets by $\triangle$.  

\begin{lem}\label{lem:1n}
Suppose $n\geq 4$.  Let $\Pi = K(S)\cup K(T)$ be pattern-Knuth closed where $S,T\in \SYT(n)$.  Assume $n-1\notin \Des(S)$ and $n-1\in \Des(T)$ and $1\in \Des(S)\triangle \Des(T)$.  Then there exists some $\pi\in K(S)$ and $\sigma\in K(T)$ with
$\pi_n = n$ and $\sigma_1 = n$.  
\end{lem}

\begin{proof}
We first prove the existence of such a $\pi\in K(S)$.  Choose $\pi\in K(S)$ that maximizes $i$ where $\pi_i=n$.  Towards a contradiction, assume $i<n$ so that $\pi_{i+1}\leq n-2$ since $n-1\not\in\Des(S)$. Now add $(n-1)^-$ in position $i$ and use it to interchange $n$ and $\pi_{i+1}$ via a Knuth move to obtain
$$\rho:=\ldots, n-1, \ldots, (n-1)^-,\pi_{i+1},n,\ldots\in \fSb_{n+1}(\Pi).$$
Let $x$ be such that $\rho-x\in \Pi$.  We claim that $\rho-x\notin K(T)$.  For if this were to occur we must take $x=1$ or $2$ since $1\in \Des(S)\triangle \Des(T)$. But because $n\geq 4$, we would then have $n-1\notin \iDes(\rho-x)$ whereas $n-1\in \Des(T)$.  We conclude that $\rho-x\in K(S)$.  

Next observe that $x\neq n$ as otherwise $n-1\in \iDes(\rho-x)$ whereas $n-1\notin \Des(S)$. Consequently, $n$ sits in position $i+1$ or $i+2$ in $\rho-x\in K(S)$ depending on whether $x$ is to the left or the right of $n$ in $\rho$ respectively. This contradicts our choice of $\pi$ proving the first claim in this lemma.  

The above argument, when applied to $\Pi^r = K(S^r)\cup K(T^r)$ establishes the second claim and completes our proof.  
\end{proof}

\begin{proof}[Proof of Theorem~\ref{thm:pairs}]

We first show that (ii) implies (iii).   As $\Pi$ is swap closed, Lemma~\ref{lem:swpclosed} together with the fact that Knuth classes are $i$-descent consistent imply that $\Pi= D_J^{-1}\cup D_L^{-1}$ for some $J,L\sbe[n-1]$.  If $J\neq L$ it further follows that $K(S) = D_J^{-1}$ and $K(T) = D_L^{-1}$.  Hence $J, L$ are as given in Lemma~\ref{lem:ssh}.  If $J=L$ then $S$ and $T$ are an $i$-descent-complete pair and $J,L$ are given by Lemma~\ref{lem:doubles}.  
The implication that (iii) implies (iv) also follows from these two lemmas.

Now assume (iv) with the goal of showing (i).  It follows from Lemmas~\ref{lem:ssh} and~\ref{lem:doubles} that $\Pi = D_J^{-1}\cup D_L^{-1}$ for some $J,L\sbe[n-1]$.  As sets of the form $D_J^{-1}$ are pattern-Knuth closed by Lemma~5.7 in \cite{hps:paq}, (i) follows from Proposition~\ref{union}. 

It remains to show that (i) implies (ii).  Observe that when $\Des(S)=\Des(T)$ the result follows from Theorem~\ref{thm:swpclosed}.  Therefore we may assume $\Des(S) \neq \Des(T)$.  In light of Lemmas~\ref{lem:ssh} and \ref{lem:swpclosed} it suffices to show that $S$ and $T$ are superstandard hooks.  We proceed by induction on $n$.  Since all tableaux are superstandard hooks when $n\leq 3$ we take $n\geq 4$.

First assume that $n-1$ is in neither $\Des(S)$ nor $\Des(T)$. By (repeated application of) Lemma~\ref{swpright} we have
\beq\label{eq:ds in Pi}
\overdoublerightarrow{\pi}\in \Pi
\eeq
for each $\pi\in \Pi$.  By Lemma~\ref{pkcinduct}, $\widehat\Pi$ is pattern-Knuth closed. It also follows from (\ref{eq:ds in Pi}) that $n$ must be in the top row of either $S$ or $T$.  We consider two cases.  If $n$ is in the top row of both $S$ and $T$ then it is the last element of both $\rho(S)$ and $\rho(T)$.  By Lemma~\ref{lem:RS and ds} see that 
\beq\label{eq:pi hat}
\widehat\Pi = \widehat{K(S)} \cup \widehat{K(T)} = K(\widehat S)\cup K(\widehat T).
\eeq
By induction we conclude that  $\widehat S$ and $\widehat T$ are superstandard hooks. As $n$ is in the top rows of $S$ and $T$ we further see $S$ and $T$ are hook shape.  Finally, Lemma~\ref{PairsHooks} implies that $S$ and $T$ are both superstandard hooks. 

Next assume that $n$ is in the top row of $S$ but not $T$.  Set $\rho = \widehat{\rho(T)},n$ so that $\rho\in \Pi$ by (\ref{eq:ds in Pi}).  Let $R$ be the insertion tableau of $\rho$.  Clearly $R = S$  or  $T$, but as  $n$ is in the top row $R$ we conclude $R=S$.  Additionally, notice that $\rho$ is the reading word for $R$ and that $R$ is obtained by moving $n$ in $T$ to the top row of $T$. As $n-1\notin \Des(T)$ it follows that $\Des(T) = \Des(R) = \Des(S)$.  This contradicts the fact that $\Des(S) \neq \Des(T)$ and so we conclude that this case cannot occur.

At this point we may assume that $n-1\in \Des(S)\triangle \Des(T)$ for if $n-1$ is in both sets then repeating the above argument on $\Pi^r = K(S^r)\cup K(T^r)$ disposes of that case.  In fact, we can  assume even more.  As Knuth moves commute with complementation we see that $\Pi^c$ is the union of Knuth classes and is pattern-Knuth closed.  Therefore the above argument, when applied to $\Pi^c$, deals with the cases when  $1\notin \Des(S)\triangle \Des(T)$.  In what remains we assume $1,n-1\in \Des(S)\triangle\Des(T)$.  

Without loss of generality assume $n-1\notin \Des(S)$ and $n-1\in \Des(T)$.  It now follows from Lemma~\ref{lem:1n} that there exists $\pi\in K(S)$ and $\sigma\in K(T)$ so that $\pi_n = n$ and $\sigma_1= n$.  Hence Lemma~\ref{lem:RS and ds} applied to $K(S)$ and $K(T)^r$ tells us that for each $\zeta\in K(S)$ and $\xi\in K(T)$ we have
$$\overdoublerightarrow{\zeta} \in K(S)\quad\textrm{and}\quad \overdoubleleftarrow{\xi}\in K(T)$$  
and that equation~(\ref{eq:pi hat}) holds here as well. Lemma~\ref{pkcinduct} now gives us that $\widehat\Pi$ is pattern-Knuth closed.  As $1\in \Des(\widehat S)\triangle \Des(\widehat T)$ we see that $\Des(\widehat S) \neq \Des(\widehat T)$ and so we may conclude by induction that $\widehat{S}$ and $\widehat{T}$ are superstandard hooks.  As $\pi_n=n$ and $\sigma_1=n$ it follows that $S$ and $T$ are hook shape.  Finally, Lemma~\ref{PairsHooks} implies that $S$ and $T$ are superstandard hooks as needed.  
\end{proof}

 \section{Stability}
\label{sta}

In Question~7.1 in~\cite{hps:paq}, the authors ask if $Q_n(\Pi)$ being symmetric or Schur nonnegative for  $n$ up to some bound would force it to continue to be so for all $n$.  We now show the converse of this question is false by showing that $Q_n(\Pi)$ can be Schur nonnegative for all sufficiently large $n$ without being so for some smaller value of $n$.  In particular, we show this is true for
\beq
\label{Pi0}
\Pi_0=K(3,1,1)-K(P_0)
\eeq
where
$$
P_0=\raisebox{4mm}{\begin{ytableau}
1&2&4\\
3\\
5
\end{ytableau}}\ .
$$

We need the following result. 
\ble[\cite{hps:paq}]
\label{pkc:M}
The set $\Pi$ is pattern-Knuth closed if and only if $\fS_n(\Pi)$ is a union of Knuth classes for $n\le M+1$ where $M$ is the maximum length of a permutation in $\Pi$.\hqed
\ele

We also need the following criterion.
\ble
\label{stable}
Given $\Pi$ and $\Pi'$  nonempty sets of permutations we let
$$
M=\max\{\#\pi \mid \pi\in\Pi\cup\Pi'\}
$$
where $\#\pi$ is the length of $\pi$.  If there is an $N\ge M$ such that
$$
\fS_N(\Pi)=\fS_N(\Pi')
$$
then
$$
\fS_n(\Pi)=\fS_n(\Pi')
$$
for all $n\ge N$.
\ele
\bprf
It suffices to prove that if $\si\in\fS_n$ contains a copy of some $\pi\in\Pi$ then $\si$ also contains a $\pi'\in\Pi'$ as the converse statement follows by symmetry.  Since $n\ge N\ge M$, there is a subsequence $\tau$ of $\si$ of length $N$ containing $\pi$.
Since $\fS_N(\Pi)=\fS_N(\Pi')$ we have that $\tau$ must also contain a copy of some $\pi'\in\Pi'$.  So $\si$ contains $\pi'$ and we are done.
\eprf

Now consider $K(3,1,1)$.  One can check by computer that $\fS_n(K(3,1,1))$ is a union of Knuth classes for $n\le 6$.  It follows from Lemma~\ref{pkc:M} the $K(3,1,1)$ is pattern-Knuth closed so that $Q_n(K(3,1,1))$ is Schur nonnegative for all $n$.

By contrast, using the computer again, we see that $Q_6(\Pi_0)$ is not even symmetric where $\Pi_0$ is defined by~\ree{Pi0}.  On the other hand,  another computer check shows that 
$\fS_7(\Pi_0)=\fS_7(K(3,1,1))$.  So, by Lemma~\ref{stable}, we conclude that $\fS_n(\Pi_0)=\fS_n(K(3,1,1))$ for $n\ge7$.  It follows from the previous paragraph that $Q_n(\Pi_0)$ is Schur nonnegative for $n\ge7$.



\end{document}